\documentclass[reqno]{amsart}
\numberwithin{equation}{section}

\usepackage{bbm}
\usepackage{comment}
\usepackage{verbatim}
\usepackage{mathrsfs}
\usepackage{mathtools}
\usepackage{latexsym}
\usepackage{epsfig}
\usepackage{amsmath}
\usepackage[usenames,dvipsnames]{xcolor}
\usepackage{graphics}
\usepackage[utf8]{inputenc}
\usepackage{amssymb}
\usepackage{amsthm}
\usepackage[alphabetic]{amsrefs}
\usepackage{amsopn}

\usepackage{hyperref}
\usepackage{tikz-cd}

\numberwithin{equation}{section}

\theoremstyle{plain}
\newtheorem{thm}[equation]{Theorem}
\newtheorem{lem}[equation]{Lemma}
\newtheorem{prop}[equation]{Proposition}
\newtheorem{cor}[equation]{Corollary}

\newtheorem*{thm*}{Theorem}
\newtheorem*{lem*}{Lemma}
\newtheorem*{prop*}{Proposition}
\newtheorem*{cor*}{Corollary}

\theoremstyle{definition}
\newtheorem{defn}[equation]{Definition}
\newtheorem*{defn*}{Definition}

\newtheorem{ex}[equation]{Example}
{}
\newtheorem{rem}[equation]{Remark}
\newtheorem*{rem*}{Remark}

{}
{}
{}
\newtheorem*{ack}{Acknowledgements}{}
\newtheorem{qn}[equation]{Question}{}
\newtheorem{slogan}[equation]{Slogan}{}

\theoremstyle{remark}
{}
{}
{}

\def\to{\longrightarrow} 

\def\FF{\mathbb{F}}

\def\NN{\mathbb{N}}

\def\QQ{\mathbb{Q}}

\def\ZZ{\mathbb{Z}}
\DeclareUnicodeCharacter{221E}{$\infty$}

\def\sfA{\mathsf{A}}

\def\sfD{\mathsf{D}}

\def\sfG{\mathsf{G}}

\def\sfJ{\mathsf{J}}
\def\sfK{\mathsf{K}}

\def\sfN{\mathsf{N}}

\def\sfR{\mathsf{R}}

\def\sfT{\mathsf{T}}

\def\sfX{\mathsf{X}}

\def\mcE{\mathcal{E}}
\def\mcF{\mathcal{F}}

\def\mfm{\mathfrak{m}}

\def\mfp{\mathfrak{p}}
\def\mfq{\mathfrak{q}}

\def\op{\mathrm{op}}

\DeclareMathOperator{\hocolim}{hocolim}
\DeclareMathOperator{\holim}{holim}
\DeclareMathOperator{\Hom}{Hom}

\DeclareMathOperator{\id}{id}
\DeclareMathOperator{\im}{im}

\DeclareMathOperator{\modu}{\mathsf{mod}}

\DeclareMathOperator{\Inj}{\mathsf{Inj}}

\DeclareMathOperator{\RHom}{\mathbf{R}Hom}

\DeclareMathOperator{\Ext}{Ext}

\DeclareMathOperator{\Coh}{Coh}
\DeclareMathOperator{\QCoh}{QCoh}

\DeclareMathOperator{\Spec}{Spec}

\DeclareMathOperator{\loc}{\mathrm{loc}}
\DeclareMathOperator{\thick}{thick}

\definecolor{internationalkleinblue}{rgb}{0.0, 0.18, 0.65}

\title[Proxy-small object]{Proxy-small objects present \\ compactly generated categories}

\author{Benjamin Briggs}
\address{Benjamin Briggs, Department of Mathematics, Huxley Building, South Kensington Campus, Imperial College London, 
London, SW7 2AZ, U.K.}
\email{b.briggs@imperial.ac.uk}

\author{Srikanth B.\ Iyengar}
\address{Srikanth B.\ Iyengar, Department of Mathematics, University of Utah, Salt Lake City, UT 84112, U.S.A.}
\email{srikanth.b.iyengar@utah.edu}

\author{Greg Stevenson}
\address{Greg Stevenson, Aarhus University, Department of Mathematics, Ny Munkegade 118, bldg. 1530
DK-8000 Aarhus C, Denmark}
\email{greg@math.au.dk}

\date{\today}

\keywords{compactly generated triangulated category, complete intersection, small object, proxy-small object}

\subjclass[2010]{18G80 (primary);  13D09 (secondary)}

\begin{document}

\begin{abstract}
\noindent  We develop a correspondence between presentations of compactly generated triangulated categories as localizations of derived categories of ring spectra and proxy-small objects, and explore some consequences. In addition, we give a characterization of proxy-smallness in terms of coproduct preservation of the associated corepresentable functor `up to base change'.
\end{abstract}

\maketitle




\section{Introduction}

Proxy-smallness is a relatively new concept that was introduced in \cite{DGI_finiteness,DGI_duality} as a weakening of smallness, also called compactness, in triangulated categories. It is general enough that the residue field of any noetherian local ring is proxy-small in the derived category, but strong enough to yield useful formulas for local cohomology. This has had striking applications, both in homotopy theory and in commutative algebra. For instance, Dwyer, Greenlees, and Iyengar proved \cite{DGI_finiteness}*{Theorem~9.4} that if a commutative noetherian ring $R$ is complete intersection then every object of $\sfD^\mathrm{b}(\modu R)$ is proxy-small and they asked if this property characterized complete intersections. This remarkable statement turns out to be true, as was proved by Pollitz \cite{pollitzci}. This opens many doors. By analogy it furnishes us with an elegant and homotopy invariant definition of complete intersection in derived settings, and it supports rather appealing, if speculative, connections to questions concerning finite generation of cohomology. The purpose of this article is to give a conceptual interpretation of proxy-smallness in analogy with the definition of compactness, that is in terms of some functor preserving coproducts. 

Let $\sfK$ be an enhanced compactly generated triangulated category. Given a compact object $k\in \sfK$ there is an equivalence between the localizing subcategory $\loc(k)$ generated by $k$ and the derived category $\sfD(E)$ where $E$ is the derived endomorphism ring $\RHom_{\sfK}(k,k)$ of $k$.  This equivalence is given by the adjoint pair of functors $(-)\otimes_E k$ and $\RHom_{\sfK}(k,-)$, and boils down to the statements $E = \RHom_{\sfK}(k,k)$ and $E\otimes_E k = k$, since both functors preserve coproducts.  For a non-compact object $P$ this statement is  false in general \textemdash{}the localizing subcategory $\loc(P)$ depends on how $P$ sits in $\sfK$ in a subtle way, and not only on its endomorphism ring. For instance, in the derived category of abelian groups $\sfD(\ZZ)$ the objects $\ZZ_p$ and $\ZZ[p^{-1}]/\ZZ$ have the same derived endomorphism ring but generate different localizing subcategories.  The slogan is: giving a proxy-small object $P$ is the same thing as giving a presentation of the compactly generated localizing subcategory $\loc(P)$ as a quotient of $\sfD(\RHom_\sfK(P,P))$ (see \ref{slogan}). 

This shows that proxy-smallness is  ubiquitous.  Moreover, as we point out, the proxy-small condition is hidden behind fundamental constructions such as passing from the derived category of a coherent ring $R$ to its ind-coherent derived category $\sfK(\Inj R)$.

Along the way, in Theorem~\ref{thm:proxysmall}, we give a purely functorial characterization of the proxy-small condition when the ambient category satisfies the telescope conjecture. An object $P$ with $E=\RHom(P,P)$ is proxy-small if and only if $\RHom(P,-)\otimes_E P$ preserves coproducts. We also consider the relative version, where one enhanced triangulated category acts on another. The paper concludes with a discussion of the closure conditions the proxy-small objects satisfy as well as some speculations and applications.

\begin{ack}
Briggs was partly supported by the European Union under Grant Agreement 10106455, and Iyengar was partly supported by National Science Foundation grant DMS-2001368. The authors thank John Greenlees for his comments, suggestions, and interest.
\end{ack}

\section{Proxy-smallness revisited}

We begin with some `recollections' on proxy-smallness and effective constructibility, both in the sense of \cite{DGI_duality}. These are `recollections' rather than recollections because the definition of proxy-smallness we use is slightly more general than the original and we give a different perspective, with some new results, on effective constructibility.

Let us fix a compactly generated stable $\infty$-category (i.e.\ an $\aleph_0$-presentable stable $\infty$-category) with homotopy category $\sfK$. The subcategory of compact objects in $\sfK$ is denoted $\sfK^\mathrm{c}$; it is, by definition, essentially small. Essentially $\sfK$ is a compactly generated triangulated category (see \cite{LurieHA}*{Remark~1.4.4.3}) and we can speak of mapping spaces, homotopy (co)limits, and so on. In particular, given any two objects there is a spectrum of maps between them (see for instance Section~2.3 of \cite{BGT} and in particular Definition~2.15) and hence any object has a ring spectrum of endomorphisms.

Little is lost if the reader prefers to think of instead fixing a differential graded, or model, category with compactly generated homotopy category. In fact, essentially all arguments take place at the level of homotopy categories. In the dg setting we would consider for any pair of objects the chain complex of maps between them and one could replace ring spectra by dg algebras.

Throughout by `localization' we will mean a reflective localization of stable $\infty$-categories, i.e.\ an exact functor with a fully faithful right adjoint (what Lurie simply calls a localization in \cite{LurieHTT}*{5.2.7.2} in the not necessarily stable setting). This is equivalent to a localization in the usual sense, i.e.\ given by inverting a suitable collection of maps in a homotopy coherent way, which moreover has a right adjoint (cf.\ \cite{kerodon} Proposition 04JL and the linked definitions). In the differential graded context this can be modelled by a Drinfeld quotient admitting a right adjoint. 

Again, given the adjunction, this boils down to something which can be checked on the homotopy categories. The right adjoint is by definition fully faithful precisely when it induces weak equivalences of mapping spaces, and so given stability, exactly when it is fully faithful on the homotopy category. This can be rephrased as asking that the components of the counit are equivalences, i.e.\ isomorphisms in the homotopy category.

For a set of objects $\sfX$ of $\sfK$ we denote by $\thick(\sfX)$ and $\loc(\sfX)$ the smallest triangulated subcategories containing $\sfX$ and closed under summands and arbitrary coproducts respectively.

\begin{defn}\label{def_set_version}
An object $P\in \sfK$ is a \emph{proxy-small object with small set of proxies} $\mathbf{p}$ if 
\[
\mathbf{p} \subseteq  \thick(P)\cap \sfK^\mathrm{c} \text{ and } P\in \loc(\mathbf{p}).
\]
We will often just say that $P$ is proxy-small without fixing a set of proxies.
\end{defn}

Throughout we regard $\mathbf{p}$ as a full subcategory of $\sfK$ and we use $i$ to denote the inclusion. Using the enhancement we may consider the ind-completion $\sfD(\mathbf{p})$ of $\thick(\mathbf{p})$ and this identifies with the full subcategory $\loc(\mathbf{p})$ of $\sfK$ with $i^*$ the corresponding fully faithful left Kan extension, i.e.\ the inclusion of $\loc(\mathbf{p})$. Details on ind-completion in general can be found in \cite{LurieHTT}*{5.3} (see in particular 5.3.5.10 and 5.3.5.11) and see \cite{LurieHA}*{1.1.3.6} for the fact that this construction preserves stability. 

\begin{rem}
If $P$ is proxy-small then $\thick(P)\cap \sfK^\mathrm{c}$ is a canonical set of proxies.
\end{rem}

\begin{rem}
If $P$ is proxy-small with set of proxies $\mathbf{p}$ then 
\[
\loc(\mathbf{p}) = \loc(P) = \loc(\thick(P)\cap \sfK^\mathrm{c}).
\]
\end{rem}

\begin{rem}
This is a small generalization of the original definition of Dwyer, Greenlees, and Iyengar \cite{DGI_duality}. Instead of insisting there be a single small proxy we allow a set of such. This has the advantage that it does not depend on a choice of proxy if we just take $\mathbf{p} = \sfK^\mathrm{c} \cap \thick(P)$.

In fact one could, and in some cases should, go further: we could also replace $P$ by a set of objects; see \ref{rem:kinj}. Everything works just as well if one takes this extra step. In our primary cases of interest we can get away with a single object $P$ and so we stick to this setting\textemdash{}replacing $P$ by a subcategory is essentially cosmetic.
\end{rem}

Let us consider the ring spectrum $E = \RHom(P,P)$ and the full subcategory $\mathbf{p}$ of $\sfK$ and form the following diagram

\begin{equation}
\label{eq:diagram}
\begin{tikzcd}
\sfD(E) \arrow[drr, shift left=3] \arrow[d, shift right]   & & \\
\loc(\mathbf{p}) \arrow[d, shift left] \arrow[rr, shift left, "i_*"] \arrow[u, shift right] & & \sfK \arrow[ll, shift left, "i^!"] \arrow[ull, shift right] \arrow[dll, shift left=3] \\
\sfD(\mathbf{p}) \arrow[u, shift left, "\wr"] \arrow[urr, shift right] &&
\end{tikzcd}
\end{equation}
consisting of the various tensor-hom adjunctions, and which commutes in the obvious ways by construction. Namely, the unlabelled left adjoints are given as follows: $\sfD(E) \to \sfK$ is $(-)\otimes_E P$ which factors via $\loc(P) = \loc(\mathbf{p})$ by virtue of preserving coproducts, $\sfD(\mathbf{p}) \to \loc(\mathbf{p})$ is the equivalence determined by the universal property of ind-completion (that is, by sending $\RHom_{\mathbf{p}}(-,p)$ to $p$ for $p \in \mathbf{p}$), and $\sfD(\mathbf{p}) \to \sfK$, which we denote by $(-)\otimes_\mathbf{p} \mathbf{p}$, is also determined by the universal property.

\begin{rem}
In most cases of interest it is possible to pick a single proxy and take $\mathbf{p} = \{p\}$. In this case we can just deal with $B = \RHom(p,p)$, for instance $\sfD(\mathbf{p}) \cong \sfD(B)$ and $\sfD(B) \to \sfK$ is the functor $(-)\otimes_B p$ which factors via $\loc(p)$. The reader may find it helpful to keep this case in mind.
\end{rem}

For the rest of this section $P$ is a proxy-small object in $\sfK$ with a set of proxies $\mathbf{p}$. We are primarily concerned with understanding the functor
\begin{equation}
\label{eq:gamma}    
\Gamma = i_*i^! \cong \RHom_\sfK(\mathbf{p},-)\otimes_\mathbf{p} \mathbf{p}.
\end{equation}
To this end let us begin by understanding the composite $\sfD(E) \to \sfD(\mathbf{p})$.

\begin{lem}
\label{lem:tech}
The functor $\sfD(E) \to \sfD(\mathbf{p})$ is naturally isomorphic to
\[
-\otimes_E \RHom_\sfK(\mathbf{p},P)
\]
where $\RHom_\sfK(\mathbf{p},P)$ is the $ E^\op\otimes \mathbf{p}$-module given by restricting $\RHom_\sfK(-,-)$ to the indicated objects. 
\end{lem}
\begin{proof}
The composite $\RHom_\sfK(\mathbf{p}, -\otimes_E P):\sfD(E) \to \sfD(\mathbf{p})$ preserves coproducts since $\mathbf{p}$ consists of compacts. As we vary $M\in \sfD(E)$, by tracing the image of the identity under the chain of isomorphisms
\begin{align*}
\RHom_K(M\otimes_E P,M\otimes_E P) &\cong \RHom_E(M, \RHom_\sfK(P, M\otimes_E P)) \\
&\cong \RHom_E(M, \RHom_\mathbf{p}(\RHom_\sfK(\mathbf{p},P), \RHom_\sfK(\mathbf{p}, M\otimes_E P))) \\
&\cong \RHom_\mathbf{p}(M\otimes_E \RHom_\sfK(\mathbf{p},P), \RHom_E(\mathbf{p}, M\otimes_E P)),
\end{align*}
where the third isomorphism is fully faithfulness of $\RHom_\sfK(\mathbf{p},-)$ on $\loc(\mathbf{p})$ and the others are adjunction isomorphisms, we obtain a natural map
\[
-\otimes_E \RHom_\sfK(\mathbf{p},P) \to \RHom_\sfK(\mathbf{p}, -\otimes_E P).
\]
This is evidently an isomorphism at $E$, and both functors preserve coproducts, so this natural transformation is an isomorphism as claimed.
\end{proof}

\begin{rem}
It follows that the right adjoint functor $\sfD(\mathbf{p}) \to \sfD(E)$ is given by
\[
\RHom_\mathbf{p}(\RHom_\sfK(\mathbf{p},P),-).
\]
\end{rem}

We now prove the key fact concerning proxy-smallness.

\begin{thm}
\label{thm:loc}
Assume $P$ is proxy-small. The functor $(-)\otimes_E P\colon \sfD(E) \to \loc(P)$ is a localization. Moreover it preserves products.
\end{thm}

\begin{proof}
We verify that the functor $\RHom_\sfK(P,-)\colon \loc(P)\to \sfD(E)$ is fully faithful and right adjoint to $(-)\otimes_E P$. Our starting point is the observation that $\RHom_\sfK(P,-)$ induces an equivalence from $\thick(P)$ to $\sfD(E)^\mathrm{c}$ with inverse $(-)\otimes_E P$. In particular, $\mathbf{p} \subseteq \thick(P)$ embeds fully faithfully into $\sfD(E)$ which induces a colimit and compact preserving embedding $\lambda\colon \loc(\mathbf{p}) \to \sfD(E)$. While the localizing subcategory $\loc(P)$ must be taken in $\sfK$, we note that $\loc(\mathbf{p})$ is unambiguous here: $\mathbf{p}$ consists of objects that are compact in both $\sfD(E)$ and $\sfK$, and so it does not matter where we form the localizing subcategory $\mathbf{p}$ generates. By Brown representability $\lambda$ has a right adjoint $\rho$. Of course, this is overkill: the right adjoint is just the obvious restriction functor (up to identifying $\loc(\mathbf{p})$ and $\sfD(\mathbf{p})$).

We claim that $\rho \cong (-)\otimes_E P$. Since both of these functors preserve colimits they are determined by their value at $E$. We use the identification of $\thick(P)$, where the thick closure is taken in $\loc(\mathbf{p}),$ with $\sfD(E)^\mathrm{c}$ again. The latter contains $\lambda(\mathbf{p})$ and so
\[
\rho(E) = \RHom_{\sfD(E)}(-,E)\vert_{\lambda(\mathbf{p})} \cong \RHom_\sfK(-,P)\vert_{\mathbf{p}}.
\]
Now the representable functor on the right corresponds to $P\in \loc(P)=\loc(\mathbf{p})$ under the equivalence $\sfD(\mathbf{p}) \cong \loc(\mathbf{p})$, and we can conclude that $\rho \cong (-)\otimes_E P$. It follows that the right adjoint $\RHom_\sfK(P,-)$ of $(-)\otimes_E P$ is fully faithful, because the left adjoint $\lambda$ of $\rho$ is so, showing that the latter is a localization. Being a right adjoint $\rho \cong (-)\otimes_E P$ preserves products as claimed.
\end{proof}

\begin{rem}
The theorem gives a rather pleasant perspective on proxy-smallness: Given an object $X$ of $\sfK$ with derived endomorphism ring $E$, there is always an adjunction
\[
\begin{tikzcd}
\sfD(E) \arrow[r, shift left] & \loc(X) \arrow[l, shift left]
\end{tikzcd}
\]
where $\sfD(R) \to \loc(X)$ is given by $(-)\otimes_R X$. However, there is no reason for this to be a localization. Proxy-smallness of $X$ guarantees that it is.
\end{rem}

This alone is already interesting in a number of cases.

\begin{cor}
Let $(R,\mfm,k)$ be a commutative noetherian local ring and let $E$ denote the derived endomorphism ring of $k$. Then $\sfD(E) \to \sfD_{\{\mfm\}}(R) = \loc(k)$ is a localization.
\end{cor}
\begin{proof}
The residue field $k$ is always proxy-small, with proxy given by, for example, the Koszul complex on $\mfm$.
\end{proof}

\begin{ex}
Consider $R=k[x]/(x^2)$. In this case, $E = k[\theta]$ with $\theta$ in (homological) degree $-1$ and no higher homotopical structure (i.e.\ $E$ is formal). We have 
\[
\sfD^\mathrm{b}(R) = \thick(k) \subseteq \loc(k) = \sfD(R)
\]
and a fully faithful embedding of $\sfD(R)$ into $\sfK(\Inj R)$ the ind-coherent sheaves on $\Spec R$ given by taking $K$-injective resolutions. The compact objects of $\sfK(\Inj R)$ are the image of $\sfD^\mathrm{b}(R)$ under this embedding and so an injective resolution of $k$ is a compact generator. Thus there is an equivalence of $\sfK(\Inj R)$ with $\sfD(E)$ which identifies the above embedding with $\RHom(k,-)\colon \sfD(R) \to \sfD(E)$. The left adjoint to this embedding is the localization $\sfD(E) \to \sfD(R) = \sfD_{\{\mfm\}}(R)$.
\end{ex}

The remarks in the previous example extend to any artinian commutative local ring, and in fact it is more general still.

\begin{ex}
\label{ex:stablederived}
Let $R$ be a coherent ring and $G$ a classical generator for $\sfD^\mathrm{b}(\modu R)$, that is to say, $\thick(G) = \sfD^\mathrm{b}(\modu R)$.  Then $G$ is certainly proxy-small: $\thick(G)$ contains all perfect complexes and hence $\loc(G) = \sfD(R)$, since the latter is generated by $R$. 

Hence, letting $E$ denote the derived endomorphism ring of $G$, we see $\sfD(R)$ is a localization of $\sfD(E)$ and $\sfD(E)^\mathrm{c} \cong \sfD^\mathrm{b}(\modu R)$. Thus $\sfD(E) \cong \sfK(\Inj R)$ and this localization is the usual one, namely the functor inverting quasi-isomorphisms,  up to equivalence.

We note that not all rings, not even commutative ones, have classical generators.
\end{ex}

\begin{rem}
\label{rem:kinj}
Allowing $P$ to consist of a set of objects, we can take $P = \sfD^\mathrm{b}(\modu R)$ for any coherent ring $R$ (or scheme) and the result can be identified with Krause's recollement \cite{KrStab} (as extended to coherent rings by \v S\v tov\'i\v cek \cite{stovicek2014purity}) connecting $\sfK(\Inj R)$ and $\sfD(R)$ essentially for free from the formalism of proxy-smallness. We are grateful to \v S\v tov\'i\v cek for pointing this out.
\end{rem}

Using Theorem~\ref{thm:loc} we can give an efficient proof of effective constructibility as in \cite{DGI_duality}*{Theorem~4.10}. The functor $\Gamma$ is the one from \eqref{eq:gamma}.

\begin{prop}
\label{prop:effective}
There is a natural isomorphism
\[
\Gamma \cong \RHom_\sfK(P,-)\otimes_E P.
\]
\end{prop}
\begin{proof}
We know, since $\mathbf{p}$ consists of compact objects, there is an isomorphism $\Gamma \cong \RHom_\sfK(\mathbf{p},-)\otimes_\mathbf{p} \mathbf{p}$. Hence, for $X\in \sfK$ we have
\[
\begin{tikzcd}
\RHom_\sfK(P,X)\otimes_E P \arrow[r, "\sim"] & \RHom_\sfK(P,X) \otimes_E \RHom_\sfK(\mathbf{p},P) \otimes_\mathbf{p} \mathbf{p} \arrow[d, "\circ \otimes_\mathbf{p} \mathbf{p}"] \\
\Gamma X & \arrow[l, "\sim"'] \RHom_\sfK(\mathbf{p},X) \otimes_\mathbf{p} \mathbf{p}
\end{tikzcd}
\]
where the top horizontal isomorphism comes from factoring $\sfD(E) \to \sfK$ via $\sfD(\mathbf{p})$ and applying Lemma~\ref{lem:tech}. So we just need to check that the vertical map is an isomorphism, and of course it suffices to show that
\[
\varepsilon\colon \RHom_\sfK(P,X) \otimes_E \RHom_\sfK(\mathbf{p},P) \to \RHom_\sfK(\mathbf{p},X)
\]
is an isomorphism. One can see this directly from chasing diagram~(\ref{eq:diagram}) using Lemma~\ref{lem:tech} and Theorem~\ref{thm:loc}. 
\end{proof}

\begin{ex}
As the reader may have already divined from the setup it is not necessarily the case that $\sfD(E)\to \sfK$ is fully faithful. In fact, it is not even necessarily conservative. 

Consider $\sfK = \sfD(\ZZ)$ and let $P$ be the injective envelope of $\FF_p$. Then $E = \RHom_\ZZ(P,P)$ is concentrated in degree $0$ and isomorphic to the $p$-adic integers $\ZZ_p$. Note that $P$ is proxy-small (e.g.\ with small proxy $\FF_p$ cf.\ Example~\ref{ex:ptorsion}). One can easily check that $\QQ_p$, the function field of $\ZZ_p$, is killed by 
\[
-\otimes_{\ZZ_p} P\colon \sfD(\ZZ_p) \longrightarrow \sfD(\ZZ)
\]
and so this functor is not fully faithful.
\end{ex}

\section{Detecting smashing localizations}

We continue the notation of the previous sections, with the caveat that we no longer assume that $P$ is necessarily proxy-small. 

If $P$ is proxy-small then the functor $\RHom_\sfK(P,-)\otimes_E P$ preserves coproducts, even though $\RHom_\sfK(P,-)$ does not if $P$ is not compact. It is then natural to ask what we can deduce about a general $P$ with the property that $\RHom_\sfK(P,-)\otimes_E P$ preserves coproducts. 

Set $F = -\otimes_E P$ and $G= \RHom_\sfK(P,-)$ for brevity, and assume that $FG$ preserves coproducts. We consider these functors, as in diagram~\eqref{eq:diagram}, as an adjunction relating $\sfD(E)$ and $\sfK$. The counit of this adjunction is denoted $\varepsilon \colon FG \to \id_{\sfK}$. The following lemma is implicit in said diagram, but let us make it explicit.

\begin{lem}\label{lem:thefirst}
The essential image of $FG$, denoted $\im FG$, is contained in $\loc(P)$.
\end{lem}
\begin{proof}
It is clear that $\im FG \subseteq \im F$. We have $FE = P$ and $F$ preserves coproducts and is exact so $\im F \subseteq \loc(P)$ by the usual arguments.
\end{proof}

Now let us use that $FG$ preserves coproducts.

\begin{lem}
\label{lem:thesecond}
The map $\varepsilon_{FG}\colon (FG)^2 \to FG$ is an isomorphism.
\end{lem}

\begin{proof}
Since $FG$ preserves coproducts 
\[
\sfN = \{X \in \sfK \mid FG(X) \stackrel{\varepsilon}\to X \text{ is an isomorphism}\}
\]
is a localizing subcategory. On $P$ we get 
\[
\varepsilon_P \colon FG(P) = \RHom_\sfK(P,P)\otimes_{E} P \stackrel{\sim}{\to} P
\]
and so $\loc(P) \subseteq \sfN$. The statement then follows from Lemma~\ref{lem:thefirst}.
\end{proof}

Since $\sfK$ is compactly generated $\loc(P)$ is well-generated \cite[Theorem 4.4.9]{NeeTri} and so by Brown representability \cite[Theorem 8.3.3]{NeeTri} there is for all $X$ a localization triangle
\begin{equation}\label{loctriangle}
\Gamma_P X \to X \to L_P X
\end{equation}
with $\Gamma_P X $ in $\loc(P)$ and $L_P X$ in $\loc(P)^\perp$.

\begin{lem}
\label{lem:thethird}
The natural map $FG(\varepsilon)\colon (FG)^2 \to FG$ is an isomorphism.
\end{lem}
\begin{proof}
Since $FG$ preserves coproducts 
\[
\sfN = \{X \in \sfK \mid FG(\varepsilon)\colon(FG)^2(X) \to FG(X) \text{ is an isomorphism}\}
\]
is a localizing subcategory. The map $FG(\varepsilon)$ is an isomorphism on $P$ (in fact the same one as in Lemma~\ref{lem:thesecond}) and hence on all of $\loc(P)$. On the other hand, if $X\in \loc(P)^\perp$ then $G(X)=0$ by definition and so $FG(X)=0$. Thus $\varepsilon_X = 0$ and so $FG(\varepsilon_X)\colon  0 \to 0$ is an isomorphism and we see $\loc(P)^\perp \subseteq \sfN$. Thus $\sfN$ contains $\loc(P)$ and $\loc(P)^\perp$, from which we conclude using the localization triangle \eqref{loctriangle} that $\sf N=\sfK$.
\end{proof}

To summarize: we have shown that the coproduct preserving functor 
\[
FG = \RHom_\sfK(P,-)\otimes_E P,
\]
equipped with the counit $\varepsilon\colon FG \to \id_\sfK$, has essential image $\loc(P)$ and that $FG(\varepsilon)$ and $\varepsilon_{FG}$ are isomorphisms. Thus provided we can show that $FG(\varepsilon) = \varepsilon_{FG}$ we would see that $FG$ is an acyclization functor and $\loc(P)$ is smashing.

This is true on objects of $\loc(P)$. Indeed, the naturality square at an object $X$ for $\varepsilon$ with respect to itself
\[
\begin{tikzcd}
(FG)^2(X) \arrow[r, "FG(\varepsilon_X)"] \arrow[d, "\varepsilon_{FG(X)}"] & FG(X) \arrow[d, "\varepsilon_{X}"]    \\
FG(X) \arrow[r, "\varepsilon_{X}"] & X
\end{tikzcd}
\]
shows these two transformations agree up to $\varepsilon_{X}$. In particular, if $X\in \loc(P)$ then $\varepsilon_X$ is an isomorphism (e.g.\ by the proof of Lemma~\ref{lem:thesecond}) and so we can cancel it to deduce that $FG(\varepsilon_X) = \varepsilon_{FGX}$.

Applying $(FG)^2 \to FG$ to the localization triangle \eqref{loctriangle}, and noting that $G(L_PX)=0$, we get a pair of maps of triangles
\[
\begin{tikzcd}
FG^2(\Gamma_P X) \arrow[r] \arrow[d, shift left=2.5] \arrow[d, shift right=2, "="] & (FG)^2(X) \arrow[r] \arrow[d, shift left=1.5, "FG(\varepsilon_X)"] \arrow[d, shift right=1.5, "\varepsilon_{FG(X)}",swap] & 0  \arrow[d, shift left=1.5] \arrow[d, shift right=1.5]  \\
FG(\Gamma_p X) \arrow[r] & FG(X) \arrow[r] & 0 
\end{tikzcd}
\]
from which we see $FG(\varepsilon_X) = \varepsilon_{FGX}$. Thus we have proved the following theorem.

\begin{thm}\label{thm:smashing}
Let $P$ be an object of $\sfK$, with derived endomorphism ring $E$, such that $\RHom_\sfK(P,-)\otimes_E P$ preserves coproducts. Then $\loc(P)$ is smashing with acyclization functor $\RHom_\sfK(P,-)\otimes_E P$.\qed
\end{thm}



\section{Characterization of proxy-smallness}

We use Theorem~\ref{thm:smashing} to give a characterization of proxy-small objects. Before proceeding let us give a quick reminder. We have an object $P$ of $\sfK$ with derived endomorphism ring $E$. This data induces a composite functor
\[
\sfD(E) \to \loc(P) \to \sfK
\]
with a right adjoint, and we have observed $\sfD(E) \to \loc(P)$ is not necessarily an equivalence. However, it does induce an equivalence $\sfD(E)^\mathrm{c} \stackrel{\sim}{\to} \thick(P)$ between the perfect $E$-modules and the thick subcategory generated by $P$.

We have the following `boundedness' result.

\begin{lem}\label{lem:bdd}
Let $P$ be an object of $\sfK$ with derived endomorphism ring $E$. Suppose that $\RHom_\sfK(P,-)\otimes_E P$ preserves coproducts. Then $\loc(P)\cap \sfK^\mathrm{c} \subseteq \thick(P)$.
\end{lem}

\begin{proof} 
Let $x$ be an object of $\loc(P)\cap \sfK^\mathrm{c}$. The object $\RHom_\sfK(P,x)$ of $\sfD(E)$ can be written as a filtered homotopy colimit $\hocolim_i F_i$ where each $F_i$ lies in $\thick(E)$. The left adjoint $(-)\otimes_E P$ preserves homotopy colimits and so we get isomorphisms
\[
x \cong \RHom_\sfK(P,x)\otimes_E P \cong (\hocolim_i F_i) \otimes_E P \cong \hocolim_i (F_i \otimes_E P)
\]
the first via Theorem~\ref{thm:smashing}. Because $x$ is compact in $\loc(P)$ the isomorphism from $x$ to $\hocolim_i (F_i \otimes_E P)$ factors through some $F_i \otimes_E P$. Hence $x$ is a retract of the object $F_i \otimes_E P \in \thick(P)$ and so $x$ lies in $\thick(P)$.
\end{proof}

\begin{rem}
For $P$ as above  we know $\loc(P)$ is smashing, by Theorem~\ref{thm:smashing}, and so the inclusion functor $\loc(P)\to \sfK$ sends compacts to compacts. Thus, being compact in $\loc(P)$ is the same as being compact in $\sfK$ and we do not need to distinguish these notions.
\end{rem}

All this leads to the following characterization of proxy-smallness.

\begin{thm}
\label{thm:proxysmall}
Let $P$ be an object of $\sfK$, with derived endomorphism ring $E$. Then $P$ is proxy-small if and only if the following two conditions hold:
\begin{itemize}
\item[(i)] $\loc(P) = \loc( \loc(P)\cap \sfK^\mathrm{c})$;
\item[(ii)] $\RHom_\sfK(P,-)\otimes_E P$ preserves coproducts.
\end{itemize}
\end{thm}

\begin{proof}
If $P$ is proxy-small then (i) holds by definition and (ii) holds by effective constructibility as in Proposition~\ref{prop:effective}.

Suppose, on the other hand, the object $P$ verifies (i) and (ii). By Theorem~\ref{thm:smashing} condition (ii) guarantees that $\RHom_\sfK(P,-)\otimes_E P$ is the acyclization functor for $\loc(P)$. By Lemma~\ref{lem:bdd} if $x\in \loc(P)$ is compact in $\sfK$ then $x$ already lies in $\thick(P)$. Condition (i) then ensures there are enough such compacts.
\end{proof}

\begin{rem}
It was already shown in \cite{DGI_duality}, provided $\loc(P)$ has a compact generator, that proxy-smallness of $P$ implies (i) and (ii).
\end{rem}

\begin{rem}
In fact the localization corresponding to a proxy-small object also preserves products. This was shown in Theorem~\ref{thm:loc}.
\end{rem}

\begin{cor}\label{cor:proxysmall}
Suppose that $\sfK$ satisfies the telescope conjecture, i.e.\ every smashing subcategory of $\sfK$ is generated by the compact objects it contains. Then $P$ is proxy-small if and only if $\RHom_\sfK(P,-)\otimes_E P$ preserves coproducts.
\end{cor}
\begin{proof}
If the telescope conjecture holds for $\sfK$ then condition (ii) already implies condition (i) by Theorem~\ref{thm:smashing}.
\end{proof}

\begin{cor}\label{cor:slogan}
Let $\sfJ$ be a localizing subcategory of $\sfK$ satisfying $\sfJ = \loc(\sfJ \cap \sfK^\mathrm{c})$. If $E$ is a ring spectrum and $\pi\colon\sfD(E) \to \sfJ$ is a localization then $\pi(E)$ is proxy-small in $\sfK$. 
\end{cor}

\begin{proof}
Being a left adjoint $\pi$ preserves coproducts, so it necessarily of the form $-\otimes_E \pi(E)$, with right adjoint given by $\RHom_\sfK(\pi(E),-)$. By the assumption that $\pi$ is a localization we have a natural isomorphism $\RHom_\sfK(\pi(E),-)\otimes_E \pi(E) \cong \id_\sfJ$. Since the inclusion $i_*\colon \sfJ\to \sfK$ preserves compacts, it has a coproduct preserving right adjoint $i^!$. Then $\RHom_\sfK(\pi(E),-)\otimes_E\pi(E)  =i_*(\RHom_\sfJ(\pi(E),i^!(-))\otimes_E\pi(E))= i_*i^!$ evidently preserves coproducts, and $\pi(E)$ is proxy-small by Theorem~\ref{thm:proxysmall}.
\end{proof}

\begin{rem}
One can also prove the preceding result by checking the definition of proxy-smallness directly using that $\sfJ = \loc(\pi(E))$, because $\pi$ is a localization, and applying Lemma~\ref{lem:bdd}.
\end{rem}

\begin{rem}
Choosing a set of small proxies $\mathbf{p}$ is the additional data of fixing another presentation of $\sfJ$ as a derived category, namely $\sfJ \cong \sfD(\mathbf{p})$.
\end{rem}

\begin{slogan}
\label{slogan}
Proxy-small objects are presentations of localizing subcategories generated by compacts: giving a proxy-small object $P$ is equivalent to giving a thick subcategory $\sfR$ of the compact objects $\sfK^\mathrm{c}$, a ring spectrum $E$, and a localization $\sfD(E) \to \loc(\sfR)$.

Let us justify this. We saw already, via Theorem~\ref{thm:loc}, that given a proxy-small $P$ with derived endomorphism ring $E$ we get a  localization $\sfD(E) \to \loc(P)$ and (by definition) $\loc(P)$ is generated by objects of $\sfK^\mathrm{c}$. Going in the other direction is precisely Corollary~\ref{cor:slogan}.
\end{slogan}

Here is another consequence of Theorem~\ref{thm:proxysmall}.

\begin{cor}
Let $R$ be a commutative noetherian ring, and $\{R\to S_\lambda\}_\lambda$ a family of flat maps of rings such that the induced map $\sqcup_{\lambda} \Spec S_\lambda\to \Spec R$ is onto. For an object $P$ in $\sfD^\mathrm{b}(\modu R)$, if $S_\lambda\otimes_RP$ is proxy-small in $\sfD(S_\lambda)$ for all $\lambda$, then $P$ is proxy-small.
\end{cor}

\begin{proof}
To begin with, we claim a map $f$ in $\sfD(R)$ is an isomorphism if $S_\lambda\otimes_R f$ is an isomorphism in $\sfD(S_\lambda)$ for each $\lambda$.

Indeed, by considering the mapping cone of $f$, it suffices to verify that for an object $M\in\sfD(R)$, if $S_\lambda\otimes_R M=0$ in $\sfD(S_\lambda)$ for all $\lambda$, then $M=0$ in $\sfD(R)$.

Fix $\mfp$ in $\Spec R$ and pick an $S_\lambda$ and a $\mfq\in \Spec S_\lambda$ such that $\mfq\cap R=\mfp$. Then, with $k(\mfp)$ and $k(\mfq)$ denoting the residue fields of $R_\mfp$ and $S_\mfq$, respectively, we get isomorphisms
\[
 k(\mfq)\otimes_S (S_\lambda\otimes_RM)
    \cong k(\mfq)\otimes_R M 
        \cong k(\mfq)\otimes_{k(\mfp)}(k(\mfp)\otimes_R M)\,.
\]
Thus we obtain that $k(\mfp)\otimes_R M=0$. Since $\mfp$ was arbitrary, it follows that $M=0$, as claimed; see, for instance, \cite[Lemma~2.12]{NeeChro}.

Set $E=\RHom_{R}(P,P)$. Since $R$ is noetherian and the $R$-module $H(P)$ is finitely generated, for any flat map $R\to S$, and $M$ in $\sfD(R)$, the natural map is an isomorphism:
\[
S\otimes_R \RHom_R(P, M ) \xrightarrow{\ \cong\  } \RHom_S(S\otimes_RP, S\otimes_RM)\,.
\]
We write $L^R_P(-)$ for the functor $\RHom_R(P,-)\otimes_EP$ on $\sfD(R)$. Given the isomorphism above, for any $M$ in $\sfD(R)$ and flat map $R\to S$ one gets a natural isomorphism:
\[
S\otimes_R L^R_P(M) \longrightarrow L^S_{S\otimes_RP}(S\otimes_RM)\tag{$\ast$}\,.
\]

Now we are ready to verify that $P$ is proxy-small: Since the telescope conjecture holds for $\sfD(R)$, by \cite[Corollary~3.4]{NeeChro}, it suffices to verify that for any family $\{M_i\}$ of $R$-complexes the natural map 
\[
f\colon \bigoplus_i L^R_P(M_i) \longrightarrow L^R_P(\bigoplus_i M_i)    
\]
is an isomorphism; see Corollary~\ref{cor:proxysmall}. Given the discussion above, it suffices to verify that  
$S_\lambda\otimes_Rf$ is an isomorphism for each $\lambda$. Applying ($\ast$) one gets that
\[
S_\lambda\otimes_Rf\colon 
    \bigoplus_i L^{S_\lambda}_{S_\lambda\otimes_RP}(S_\lambda\otimes_RM_i))
         \longrightarrow L^{S_\lambda}(\bigoplus_i (S_\lambda\otimes_RM_i))\,.
\]
This is an isomorphism by Theorem~\ref{thm:proxysmall} because $S_\lambda\otimes_RP$ is proxy-small, by hypothesis.
\end{proof}

 It is illustrative to further specialise the preceding corollary; it recovers \cite[Propositions~5.3 and 5.5]{Letz:2021}. 

 \begin{cor}
 Let $R$ be a commutative noetherian ring and  $P$ in $\sfD^\mathrm{b}(\modu R)$. The following conditions are equivalent:
 \begin{itemize}
     \item[\rm(1)] $P$ is proxy-small;
     \item[\rm(2)] $P_\mfp$ is proxy-small in $\sfD(R_\mfp)$ for each $\mfp\in\Spec R$;
     \item[\rm(3)] $\widehat R\otimes_RP$ is proxy-small in $\sfD(\widehat R)$ where $\widehat R$ is the completion of $R$ with respect to an ideal in the Jacobson radical of $R$;
     \item[\rm(4)] $S\otimes_R P$ is proxy-small in $\sfD(S)$ for some faithfully flat map of rings $R\to S$. \qed
 \end{itemize} 
 \end{cor}

Returning to the general context, when $\sfK$ has a compact generator each compactly generated localizing subcategory comes with a canonical presentation. Let $\mathbf{p}$ be a set of compact objects of $\sfK$ and set 
\[
\sfJ^\mathrm{c} = \thick(\mathbf{p}) \subseteq \loc(\mathbf{p}) = \sfJ.
\]
We continue to denote by $i_*$ and $i^!$ the fully faithful inclusion $\sfJ \to \sfK$ and its right adjoint which is a localization. This gives rise to the acyclization functor $\Gamma = i_*i^!$ as before.

\begin{prop}\label{prop:thesource}
Suppose that $\sfK$ has a compact generator $g$. The object $\Gamma g$ is proxy-small. Moreover, one can take $\mathbf{p}$ as a set of small proxies.
\end{prop}
\begin{proof}
By assumption we have $\thick(g) = \sfK^\mathrm{c} \supseteq \sfJ^\mathrm{c}$. Applying $\Gamma$ we obtain
\[
\thick(\Gamma g) \supseteq \Gamma \thick(g) = \Gamma(\sfK^\mathrm{c}) \supseteq \Gamma(\sfJ^\mathrm{c}) = \sfJ^\mathrm{c} \supseteq \mathbf{p}.
\]
By construction $\loc(\Gamma g) = \sfJ = \loc(\mathbf{p}) $ and so $\Gamma g$ is proxy-small, with set of proxies $\mathbf{p}$.
\end{proof}

\begin{rem}
Even if $\sfK$ does not have a generator we can, using the expanded definition of proxy-smallness of Remark \ref{def_set_version}, consider $\Gamma\sfK^\mathrm{c}$ with set of small proxies $\sfK^\mathrm{c}$.
\end{rem}

This gives the claimed canonical presentation which is, it seems, new as well as being rather striking.

\begin{cor}\label{cor:thesource}
Suppose that $\sfK$ has a compact generator $g$. Then $\sfJ$ is a localization of $\sfD(\RHom_\sfK(\Gamma g, \Gamma g))$.\qed
\end{cor}

It is interesting to understand this phenomenon more thoroughly. In analogy with the situation in algebraic geometry one should perhaps think of this as some form of completion.



\section{The relative setting}\label{sec:relative}

In this section we discuss a version of proxy-smallness in the setting of a closed monoidal triangulated category acting on another triangulated category. Our main motivation is to capture phenomena in algebraic geometry and representation theory where typically one needs additional closure conditions, coming from the action, to realize the connection between supports and thick subcategories.

Let $\sfT$ be a monoidal $\aleph_0$-presentable stable $\infty$-category with an associative and unital left action $\sfT \times \sfK \stackrel{\ast}{\to} \sfK$ on the $\aleph_0$-presentable stable $\infty$-category $\sfK$. We assume that the monoidal structure and action preserve colimits in each variable. In particular, at the level of homotopy categories we have an action of a monoidal triangulated category in the sense of \cite{StevensonAction} (except without the insistence on symmetry of the monoidal structure), and as before it is sufficient for the reader to just have this in mind. A prototypical example is the action of $\sfT$ on itself via the monoidal structure.

Given a full subcategory $\sfA$ of $\sfK$ we denote by $\loc^\ast(\sfA)$ and $\thick^\ast(\sfA)$ the smallest localizing $\sfT$-submodule of $\sfK$ containing $\sfA$ and the smallest thick $\sfT^\mathrm{c}$-submodule of $\sfK$ containing $\sfA$ respectively. These are given by closing under the action, e.g.\ taking $\{X\ast A \mid X\in \sfT, A\in \sfA\}$ and then taking the smallest localizing subcategory containing these objects gives $\loc^\ast(\sfA)$.

\begin{defn}\label{defn:relative}
An object $P\in \sfK$ is a \emph{$\ast$-proxy-small object with small set of proxies} $\mathbf{p}$ if 
\[
\mathbf{p} \subseteq \sfK^\mathrm{c} \cap \thick^\ast(P) \text{ and } P\in \loc^\ast(\mathbf{p}).
\]
As before, we will often just say that $P$ is $\ast$-proxy-small without fixing a set of proxies.
\end{defn}

\begin{rem}
The above definition is equivalent to asking that $\sfT^\mathrm{c}\ast P = \{x\ast P \mid x\in \sfT^\mathrm{c}\}$ is proxy-small in the absolute sense.
\end{rem}

\begin{rem}
As in the absolute setting it makes perfect sense to replace the object $P$ by a full subcategory and this only results in cosmetic changes.
\end{rem}

This again gives rise to presentations of $\loc^\ast(P)$, that are now $\sfT$-linear in an appropriate sense. Recall that $\sfK$ satisfies the $\ast$-telescope conjecture if every smashing $\ast$-submodule of $\sfK$ is generated by objects of $\sfK^\mathrm{c}$.

\begin{thm}\label{thm:relative}
Let $P$ be an object of $\sfK$, let $\sfG$ be a set of compact generators for $\sfT$, and set $\sfG\ast P = \{g\ast P \mid g\in \sfG\}$. Let us moreover assume that $\sfK$ satisfies the $\ast$-telescope conjecture. Then the following statements are equivalent:
\begin{itemize}
\item[(i)] $P$ is $\ast$-proxy-small;
\item[(ii)] the functor $\RHom_\sfK(\sfG\ast P,-)\otimes_{\sfG\ast P} \sfG\ast P$ preserves coproducts;
\item[(iii)] there is a product preserving $\sfT$-linear localization $\sfD(\sfG\ast P) \to \loc^\ast(P)$ sending $\RHom_\sfK(-,P)$ to $P$.
\end{itemize}
\end{thm}
\begin{proof}
The equivalence of (i) and (ii) is a direct application of Corollary~\ref{cor:proxysmall}.

Now let us show that (i) and (iii) are also equivalent. Let $P$ be an $\ast$-proxy-small object of $\sfK$ and fix $\sfG$ as in the statement. Because $\thick(\sfG) = \sfT^\mathrm{c}$ we know $\thick(\sfG\ast P) = \thick^\ast(P)$ (cf.\ \cite{StevensonAction}*{Lemma~3.11}) and so $\sfG\ast P$ is proxy-small in the absolute sense. Theorem~\ref{thm:loc} then furnishes us with the required localization and it only remains to construct the action of $\sfT$ on $\sfD(\sfG\ast P)$ and check linearity. 

The compact objects of $\sfD(\sfG\ast P)$ are equivalent to $\thick(\sfG\ast P) = \thick^\ast(P)$ and so are canonically endowed with an action of $\sfT^\mathrm{c}$. Passing to ind-completions gives the desired action of $\sfT$ on $\sfD(\sfG\ast P)$ and linearity of the localization $\sfD(\sfG\ast P) \to \loc^\ast(P)$ then follows from the fact it is linear on generators and everything in sight preserving colimits.

Finally, suppose we are given a localization $\pi\colon \sfD(\sfG\ast P) \to \loc^\ast(P)$ as in (iii) of the statement. Then by $\sfT$-linearity $\RHom_\sfK(-,g\ast P)$ is sent  to $g\ast P$, and by Corollary~\ref{cor:slogan} the subcategory $\sfG\ast P$ of $\sfK$ is proxy-small. Using again that $\thick(\sfG\ast P) = \thick^\ast(P)$ we see $P$ is $\ast$-proxy-small.
\end{proof}

\begin{rem}
The canonical choice is $\sfG = \sfT^\mathrm{c}$. Different choices satisfying $\sfG \subseteq \sfG'$ give presentations, as in (iii),  related by $\sfT$-linear localizations.
\end{rem}

\begin{ex}\label{ex:nonaffineci}
Let $T$ be a separated noetherian regular scheme of finite Krull dimension, $\mcE$ a vector bundle on $T$, $t\in \Gamma(T,\mcE)$ a regular section, and let $X$ be the zero scheme of $t$. We take $\sfT = \sfK = \sfD(\QCoh X)$ the unbounded derived category of quasi-coherent sheaves which acts on itself via the derived tensor product. Then \cite{StevensonDuality}*{Lemma~4.1} asserts that every object $\mcF \in \sfD^\mathrm{b}(\Coh X)$ is $\otimes$-proxy-small.
\end{ex}

\begin{ex}\label{ex:nonpnilpotent}
Let $G$ be a finite group and let $k$ be a field whose characteristic divides the order of $G$. We take $\sfT = \sfK = \sfD(kG)$ acting on itself via the tensor product. Then every object of $\sfD^\mathrm{b}(\modu kG)$ is $\otimes$-proxy-small.
\end{ex}



\section{Closure conditions}

We give a brief discussion of the closure conditions the proxy-small objects satisfy. They are rather different than what one generally hopes for in the stable setting. We begin with the observation that proxy-smallness is preserved under taking coproducts.

\begin{lem}
\label{lem:coprods}
The class of proxy-small objects is closed under suspensions and arbitrary co-products.
\end{lem}

\begin{proof}
It is clear that a suspension of a proxy-small object is proxy-small. 

Let $\{P_\lambda\mid \lambda \in \Lambda\}$ be a set of proxy-small objects. Set $P = \oplus_\lambda P_\lambda$ and
$\mathbf{p}_\lambda = \loc(P_\lambda)\cap \sfK^\mathrm{c}$.  By virtue of being summands of $P$ each $P_\lambda$ is in $\thick(P)$ and hence $\mathbf{p}_\lambda \subseteq \thick(P)$ for each $\lambda$. On the other hand, $\loc(P) = \loc(P_\lambda \mid \lambda \in \Lambda)$ and $\loc(P_\lambda) = \loc(\mathbf{p}_\lambda)$ so
\[
\loc(P_\lambda \mid \lambda \in \Lambda) \subseteq \loc(\cup_\lambda \mathbf{p}_\lambda) \subseteq \loc(P).
\]
Thus $\loc(P) = \loc(\cup_\lambda \mathbf{p}_\lambda)$. We already saw $\cup_\lambda \mathbf{p}_\lambda \subseteq \thick(P)$ and so $P$ is proxy-small.
\end{proof}

However, closure under cones is rather unusual.

\begin{lem}
\label{lem:cones}
Proxy-smallness is preserved under taking cones if and only if every object is proxy-small. In particular, this implies every localizing subcategory of $\sfK$ is generated by objects of $\sfK^\mathrm{c}$.
\end{lem}

\begin{proof}
Since the full subcategory of proxy-small objects is closed under suspensions and arbitrary coproducts, it is closed under cones precisely if it is a localizing subcategory. Since the objects of $\sfK^\mathrm{c}$ are proxy-small this occurs if only if every object is proxy-small.

Suppose then that this is the case. If $\sfJ$ is a localizing subcategory then we have
\[
\sfJ = \loc(\loc(X)\mid X\in \sfJ) = \loc(\loc(X)\cap\sfK^\mathrm{c} \mid X\in \sfJ) = \loc(\sfJ \cap \sfK^\mathrm{c}),
\]
i.e.\ $\sfJ$ is generated by the compact objects it contains.
\end{proof}

\begin{ex}\label{ex:ptorsion}
Consider $\sfK = \sfD_{\{p\}}(\ZZ) = \loc(\FF_p) \subseteq \sfD(\ZZ)$. We claim that every object of $\sfK$ is proxy-small. Indeed, if $X\neq 0$ is an object of $\sfK$ then there is a triangle
\[
\begin{tikzcd}
X \arrow[r,"p"] & X \arrow[r] & X\otimes \FF_p \arrow[r] & \Sigma X
\end{tikzcd}
\]
where $X\otimes \FF_p$ is a non-zero sum of suspensions of $\FF_p$. Thus $\FF_p \in \thick(X)$ and of course $X \in \loc(\FF_p) = \sfK$.
\end{ex}

\begin{ex}\label{ex:generaltorsion}
The above example generalizes to any non-singular point on a noetherian scheme, i.e.\ if $(R,\mfm,k)$ is a regular local ring then every object of $\sfD_{\{\mfm\}}(R) = \loc(k)$ is proxy-small. One repeats the same argument using that the Koszul complex $K(\mfm)$ is isomorphic to $k$ and that $X\otimes K(\mfm) \in \thick(X)$ is both non-zero (if $X\neq 0$) and a sum of suspensions of $k$.
\end{ex}

\begin{ex}
For a local noetherian ring $(R,\mfm,k)$ it need not be the case that every object of $\loc(k)$ is proxy-small. Indeed, if $R$ is artinian then $\loc(k) = \sfD(R)$. If $R$ is not complete intersection then by \cite{pollitzci} there is some $X\in \sfD^\mathrm{b}(R)$ which is \emph{not} proxy-small.
\end{ex}

In general, the proxy-small objects are also not closed under taking summands.

\begin{ex}\label{ex:nosummands}
Consider $\sfK = \sfD(\ZZ)$. The object $\QQ$ is not proxy-small: 
\[
\thick(\QQ) = \{\oplus_{i=1}^m \Sigma^{a_i} \QQ^{b_i} \mid a_i \in \ZZ, b_i \in \NN\}
\]
has no non-zero compact object. However $\ZZ\oplus \QQ$ is proxy-small with proxy $\ZZ$.
\end{ex}

Since the proxy-small condition depends only on the thick subcategory generated by an object, it is natural to consider proxy-small thick subcategories, as in Remark \ref{def_set_version}.

\begin{prop}
    The collection of proxy-small thick subcategories in $\sfK$, ordered by inclusion, has all least upper bounds.
\end{prop}

\begin{proof}
    By the same argument as for Lemma \ref{lem:coprods}, if $\{P_\lambda\mid \lambda \in \Lambda\}$ is a family of proxy-small thick subcategories, then $\thick(P_\lambda\mid \lambda \in \Lambda)$ is proxy-small as well.
\end{proof}



\section{A perspective on complete intersections}
\label{sec:apps}

We would like to use this new conceptual understanding of proxy-smallness to our benefit. In this section we provide some first applications, present some more examples, and pose some problems.


One setting in which proxy-smallness has been studied extensively is $\sfK = \sfD(R)$ for $R$ a commutative noetherian ring. In this case $\sfD(R)$ satisfies the telescope conjecture by \cite{NeeChro} and so Theorem~\ref{thm:proxysmall} applies in full force. This lets us reformulate various results in terms of the coproduct preservation of certain functors.

As an example, let $\phi\colon R\to S$ be a map of commutative noetherian rings such that $\phi$ is flat and essentially of finite type. Set
\[
\mathrm{HH}^*(S/R,-) \coloneqq \RHom_{S\otimes_R S}(S,-)\colon \sfD(S\otimes_RS)\longrightarrow \sfD(S)\,,
\]
the Hochschild cohomology functor. Viewing $\mathrm{HH}^*(S/R)=\mathrm{HH}^*(S/R,S)$ in the enhanced sense, i.e.\ as a dg algebra or ring spectrum, and combining Theorem~\ref{thm:proxysmall} with the work of Briggs, Iyengar, Letz, and Pollitz \cite{BILP} we deduce the following corollary to their work.

\begin{cor}\label{cor:ci}
With notation as above,  $\phi$ is locally complete intersection if and only
\[
 \mathrm{HH}^*(S/R,-)\otimes_{\mathrm{HH}^*(S/R)} S \colon \sfD(S\otimes_RS)\longrightarrow \sfD(S)
\]
preserves coproducts. Moreover, if this is the case then the above gives local cohomology on $\Spec S \times_{\Spec R} \Spec S$ with support on the diagonal $\Delta$ and 
\[
-\otimes_{\mathrm{HH}^*(S)} S \colon \sfD(\mathrm{HH}^*(S/R)) \to \sfD_\Delta(S\otimes_R S)
\]
is a localization. \qed
\end{cor}

There are other settings where proxy-smallness of the diagonal appears, although this point of view does not seem to have been made explicit in the literature.

\begin{cor}\label{cor:kG}
Let $G$ be a $p$-group and $k$ a field. Then $\sfD(k(G\times G))$ is a localization of $\sfD(\mathrm{HH}^*(kG))$. 
\end{cor}

\begin{proof}
Since  $k(G\times G)\cong kG^\mathrm{e}$ is in $\thick(k)$, it follows that $kG \otimes_k k(G\times G)$ is  in $\thick(kG)$, which exhibits a non-zero compact object in $\thick(kG)$. Thus $kG$ is proxy-small over  $k(G\times G)$.  This proves the  corollary.
\end{proof}


\begin{qn}
Is there a natural description of the kernel of the localization functors $\sfD(\mathrm{HH}^*(S)) \to \sfD_\Delta(S\otimes_R S)$ and $\sfD(\mathrm{HH}^*(kG)) \to \sfD(k(G\times G))$?
\end{qn}

In fact, it seems in examples that proxy-smallness of the diagonal bimodule and finite generation of Hochschild cohomology occur together. This prompts a definition (which does not seem to have appeared in the literature yet) and some questions.

\begin{defn}\label{defn:ci}
Let $E$ be a commutative ring spectrum and let $R$ be a commutative $E$-algebra. We say that $R$ is \emph{complete intersection} (relative to $E$) if $\Delta$ is proxy-small in $\sfD(R^\op\otimes_E R)$ where $\Delta$ is $R$ viewed with the obvious bimodule structure (i.e.\ the bimodule corresponding to the identity functor).
\end{defn}

\begin{rem}
One can modify the above definition by only asking for some relative version of proxy-smallness. For instance, $\sfD(R^\op\otimes_E R)$ is monoidal and we could ask for $\otimes$-proxy-smallness of the diagonal. However, since the diagonal is the identity for the monoidal structure this turns out to be trivial. 

A good definition should involve some more restrictive closure condition, for instance proxy-smallness relative to the invertible bimodules as in Example~\ref{ex:qci}.
\end{rem}

\begin{qn}
Greenlees has given several versions of the complete intersection property for ring spectra \cite{Greenlees2018}*{Definition~14.3}. How does this definition fit into his hierarchy?
\end{qn}

\begin{qn}
If the graded-commutative ring $\mathrm{HH}^*(R/S)$ is finitely generated is $\Delta$ proxy-small? If $\Delta$ is proxy-small does it imply $\mathrm{HH}^*(R/S)$ is finitely generated?
\end{qn}

\begin{ex}\label{ex_gor_mon} It follows from work of Dotsenko, G\'elinas and Tamaroff that noncommutative Gorenstein monomial algebras are complete intersection in the sense of Definition~\ref{defn:ci}.

Let $Q$ be a finite quiver, $k$ a field, and $I$ an ideal in $kQ$ generated by paths. Assume that the monomial algebra $\Lambda= kQ/I$ is Iwanaga-Gorenstein. According to \cite[Theorem 6.3]{DGT2023} there is an explicit map $\chi\colon \Sigma^{-p} \Lambda \to \Lambda$ in $\sfD(\Lambda^\op\otimes_k \Lambda)$ such that, for all finite dimensional $\Lambda$-modules $M$ and $N$, $\Ext^*_\Lambda(M,N)= \mathrm{HH}^*(\Lambda,\Hom_k(M,N))$ is finitely generated over the subring $k[\chi]\subseteq \mathrm{HH}^*(\Lambda)$; that is, $\Lambda$ satisfies the so-called ${\bf Fg}$ condition. Varying $M$ and $N$ over the simple $\Lambda$-modules we obtain all simple $\Lambda$-bimodules as $\Hom_k(M,N)$. 
Therefore, for each simple bimodule $s$, applying the functor $\Ext_{\Lambda^\op\otimes_k \Lambda}(-,s)$ to the triangle
\[
\Sigma^{-p} \Lambda \xrightarrow{\ \chi\ }\Lambda \to \mathrm{cone}(\chi)
\]
shows that $\Ext_{\Lambda^\op\otimes_k \Lambda}(\mathrm{cone}(\chi),s)$ is finite dimensional. We can conclude that $\mathrm{cone}(\chi)$ is perfect over $\Lambda^\op\otimes_k \Lambda$. On the other hand, $\mathrm{cone}(\chi^n)$ is in $\thick(\mathrm{cone}(\chi))$ for all $n$, and we may recover the diagonal bimodule as the homotopy limit $\Lambda =  \holim(\mathrm{cone}(\chi^n))$. This shows that $\Lambda\in \loc(\mathrm{cone}(\chi))$, and so $\Lambda$ is proxy-small in $\sfD(\Lambda^\op\otimes_k \Lambda)$ with proxy $\mathrm{cone}(\chi)$.

Likewise, we can deduce that every object $M$ of $\sfD^\mathrm{b}(\Lambda)$ is proxy-small in $\sfD(\Lambda)$ with proxy $M\otimes_\Lambda\mathrm{cone}(\chi)$.

This homological behaviour is similar to that of hypersurface rings in commutative algebra, where $\chi$ is the cohomology operator of Gulliksen. In that case one can always take $p=2$, reflecting the $2$-periodicity of matrix factorizations over  hypersurface rings.

The class of Gorenstein monomial algebras includes all gentle algebras by \cite{GeissReiten}.
\end{ex}

    In other complete-intersection-like examples the diagonal bimodule is not enough, but rather the set of invertible bimodules is proxy-small in the sense of Remark \ref{def_set_version}. 

\begin{ex}\label{ex:qci}
 Let $q$ be a unit in a field $k$, and consider the quantum complete intersection
\[
\Lambda = k\langle x,y\rangle /(x^m,y^n, xy-qyx).
\]
For any automorphism $\sigma$ of $\Lambda$ there is an invertible bimodule $\Lambda_\sigma$ with underlying $k$-module $\Lambda$ and with left and right action $a\cdot m \cdot c = am\sigma(c)$ (multiplication in $\Lambda$). We will use the two commuting automorphisms $\alpha$ and $\beta$ given by $\alpha(x)=q^{-1}x$, $\alpha(y)=y$, and  $\beta(y)=qy$, $\beta(x)=x$ below.

The extension algebra $\Ext^*_\Lambda(k,k)$ of $\Lambda$ is finite over a subalgebra $k\langle u,v\rangle/(uv-q^{mn}vu)$
 \cite[5.3]{BerghOppermann}. Moreover one can compute using \cite[3.1]{BW} that, under the map 
\[
{\rm HH}^*(\Lambda,\Lambda_{\beta^m})\to {\rm HH}^*(\Lambda,k) = \Ext^*_\Lambda(k,k),
\]
the class $u$ lifts canonically to a class $\chi_u$ in ${\rm HH}^*(\Lambda,\Lambda_{\beta^m})$. Similarly, $v$ lifts to a class $\chi_v$ in ${\rm HH}^*(\Lambda,\Lambda_{\alpha^n})$. These cohomology operators commute up to a factor of $q^{mn}$, that is, they fit into a commuting square 
\[
\begin{tikzcd}[column sep = 13mm]
\Sigma^{-4}\Lambda \arrow[r, "\chi_u"] \arrow[d, "\chi_v"'] & \Sigma^{-2}2\Lambda_{\beta^m} \arrow[d, "\chi_v" ]  \\
\Sigma^{-2}2\Lambda_{\alpha^n}\arrow[r, "q^{mn}\chi_u"]  & \Lambda_{\alpha^n\beta^m}
\end{tikzcd}
\]
in $\sfD(\Lambda^\op\otimes_k \Lambda)$. Finally, we take the colimit of this diagram (that is, take the cone twice) to obtain a quantum Koszul complex $K(\chi_u,\chi_v)$ \cite{Manin}. A similar argument to Example \ref{ex_gor_mon} shows that $K(\chi_u,\chi_v)$ (and all of its twists $K(\chi_u,\chi_v)_{\alpha^i\beta^j}$) is perfect in $\sfD(\Lambda^\op\otimes_k \Lambda)$. Just as before, we can also recover $\Lambda$ (and all of its twists  $\Lambda_{\alpha^i\beta^j}$) from the objects $K(\chi_u,\chi_v)_{\alpha^i\beta^j}$ by taking a homotopy limit. Altogether, we have sketched a proof that the set of invertible bimodules $\{\Lambda_{\alpha^i\beta^j} \mid i,j\in \ZZ\}$ is proxy-small, with a set $\{K(\chi_u,\chi_v)_{\alpha^i\beta^j} \mid i,j\in \ZZ\}$ of proxies.

This example is in the spirit of Section~\ref{sec:relative}: in order for the thick closure to be large enough we have closed up with respect to the action of the derived Picard group (or just  the subgroup generated by $\Lambda_\alpha$ and $\Lambda_\beta$, in this case). One can formulate it instead by saying that the diagonal bimodule $\Lambda$ is $\otimes$-proxy-small with respect to the action of the ind-completion of $\thick(\Lambda_{\alpha^i\beta^j} \mid i,j\in \ZZ)$.
\end{ex}

\begin{ex}
A simply connected finite CW-complex $X$ is called \emph{rationally elliptic} if $\pi_n(X)\otimes_{\mathbb{Z}}\mathbb{Q}=0$ for all but finitely many $n$. The results of \cite{GHS2013} show that if $X$ is rationally elliptic then the dg algebra $C^*(X;\mathbb{Q})$ of rational cochains on $X$ has a proxy-small diagonal, and therefore is complete intersection over $\mathbb{Q}$ in the sense of \ref{defn:ci}.

Rationally elliptic spaces are ``spherically complete intersection" in the sense of \cite{GHS2013}---they can be constructed, up to rational homotopy equivalence, by starting with a point and using finitely many spherical fibrations. Therefore by \cite[Theorem 10.1]{GHS2013} they are ``endomorphism complete intersection", that is, there is a sequence of triangles of dg $C^*(X\times X;\mathbb{Q})$-modules
\[
C^*(X)=M_0\xrightarrow{\ g_1\ } \Sigma^{n_1}M_0 \to M_1,\quad\cdots\quad M_{c-1}\xrightarrow{\ g_c\ } \Sigma^{n_c} M_{c-1}\to M_c
\]
such that $M_c$ is a perfect dg $C^*(X\otimes X;\mathbb{Q})$-module. Each $M_i$ has finite dimensional total homology, and by \cite[Remark 10.2]{GHS2013} each $g_i$ has strictly positive cohomological degree. Therefore, as in Example \ref{ex_gor_mon}, we can recover each $M_{i-1}$ from $M_i$ by taking a countable homotopy colimit. It follows that $C^*(X)$ is proxy-small as a dg module over $C^*(X\times X)\simeq C^*(X)\otimes C^*(X)$ with proxy $M_c$, as claimed.

Various other complete-intersection-like conditions on a space are considered in \cite{GHS2013}. We also note that the results of \cite{GHS2013} are stated much more generally for strongly noetherian rational spaces (we have restricted to finite CW complexes only for the sake of familiarity).
\end{ex}


\subsection*{Proxy-smallness does not come cheap}

It is natural to ask if there are there are some easy conditions we could put on $P$ and $E = \RHom(P,P)$ that would guarantee $P$ is proxy-small. The following example indicates that even imposing some very strong finiteness conditions on $E$ is not  sufficient.

\begin{ex}
As we saw in Example~\ref{ex:nosummands} in $\sfD(\ZZ)$ the object $\QQ$ is not proxy-small, and so having endomorphism ring a field is not enough.
\end{ex}

\begin{ex}
Consider $R = k[x,y]/(x,y)^2$, for some field $k$, which is an artinian non-Gorenstein ring. We take for $P$ the indecomposable injective module $\Hom_k(R,k)$. Then $\RHom_R(P,P)$ is just $R$. In particular, it is discrete (i.e.\ no higher homotopical information required) and artinian.

Nevertheless, $P$ is not proxy-small. Indeed, $\thick(P)$ consists of the bounded complexes of injectives with finite dimensional cohomology and so $\thick(P) \cap \sfD^\mathrm{perf}(R)$ consists of complexes with both finite projective and injective dimension. The only such complex can be $0$, because  $R$ is not Gorenstein; see for example \cite[2.10]{Foxby_flat}. 
\end{ex}

None of this is too surprising. Indeed, proxy-smallness of $P$ depends crucially on how $P$ sits in $\sfK$ and so it is quite a subtle property.
A related question is:

\begin{qn}
What do proxy-small objects form? That is to say, if we fix a compactly generated localizing subcategory $\sfJ$ what is the right structure on the collection of proxy-small objects which generate $\sfJ$?
\end{qn}





\bibliographystyle{amsplain}

\end{document}